\documentclass[12pt]{amsart}
\usepackage[top=30truemm,bottom=30truemm,left=25truemm,right=25truemm]{geometry}
\usepackage{mathrsfs}

\usepackage{color}
\usepackage{bm}
\usepackage{amsfonts,amssymb}
\usepackage{dsfont}
\usepackage{amscd}
\usepackage{extarrows}
\usepackage{amsmath}
\usepackage{mathrsfs}
\usepackage{enumerate}
\usepackage{amscd}
\usepackage[all]{xy}
\usepackage{hyperref}
\theoremstyle{plain} 
\newtheorem{theorem}{\indent\bf Theorem}[section]

\newtheorem{conjecture}[theorem]{\indent\bf Conjecture}
\theoremstyle{definition} 

\newtheorem{thm}{Theorem}[section]
\newtheorem{cor}[thm]{Corollary}
\newtheorem{lem}[thm]{Lemma}
\newtheorem{prop}[thm]{Proposition}

\theoremstyle{definition}
\newtheorem{defn}{Definition}[section]

\theoremstyle{remark}

\newcommand{\be}{\begin{equation}}
	\newcommand{\ee}{\end{equation}}
\newcommand{\bea}{\begin{eqnarray}}
	\newcommand{\eea}{\end{eqnarray}}
\newcommand{\ben}{\begin{eqnarray*}}
	\newcommand{\een}{\end{eqnarray*}}
\newcommand{\bt}{\begin{split}}
	\newcommand{\et}{\end{split}}
\newcommand{\bet}{\begin{equation}}
	\newcommand{\mc}{\mathbb{C}}
	
	\newcommand{\ra}{\rightarrow}
	\newcommand{\beq}{\begin{equation*}}
		\newcommand{\eeq}{\end{equation*}}
	
	\newcommand{\bi}{\begin{itemize}}
		\newcommand{\ei}{\end{itemize}}
\begin{document}
		
		\title[Linear isometric invariants of bounded domains]
		{Linear isometric invariants of bounded domains}
		
		\author[F. Deng]{Fusheng Deng}
		\address{Fusheng Deng: \ School of Mathematical Sciences, University of Chinese Academy of Sciences\\ Beijing 100049, P. R. China}
		\email{fshdeng@ucas.ac.cn}
		\author[J. Ning]{Jiafu Ning}
		\address{Jiafu Ning: \ Department of Mathematics, Central South University, Changsha, Hunan 410083, P. R. China.}
		\email{jfning@csu.edu.cn}
		\author[Z. Wang]{Zhiwei Wang}
		\address{ Zhiwei Wang: \ Laboratory of Mathematics and Complex Systems (Ministry of Education)\\ School of Mathematical Sciences\\ Beijing Normal University\\ Beijing 100875\\ P. R. China}
		\email{zhiwei@bnu.edu.cn}
		\author[X. Zhou]{Xiangyu Zhou}
		\address{Xiangyu Zhou: Institute of Mathematics\\Academy of Mathematics and Systems Sciences\\and Hua Loo-Keng Key
			Laboratory of Mathematics\\Chinese Academy of
			Sciences\\Beijing\\100190\\P. R. China}
		\address{School of
			Mathematical Sciences, University of Chinese Academy of Sciences,
			Beijing 100049, China}
		\email{xyzhou@math.ac.cn}
		
		\begin{abstract}
			We introduce two new conditions for bounded domains, namely  $A^p$-completeness  and boundary blow down type,
			and show that, for two bounded domains $D_1$ and $D_2$ that are $A^p$-complete and not of boundary blow down type, if
			there exists a linear isometry from $A^p(D_1)$ to $A^{p}(D_2)$ for some real number $p>0$ with $p\neq $ even integers,
			then $D_1$ and $D_2$ must be holomorphically equivalent, where for a domain $D$, $A^p(D)$ denotes the space of $L^p$ holomorphic functions on $D$.
		\end{abstract}

		\maketitle
		
		\section{Introduction}\label{sec:intro}
		It is   known that two pseudoconvex domains (or even Stein manifolds) $ D_1$ and $ D_2$ are biholomorphic
		if and only if $\mathcal O( D_1)$ and $\mathcal O( D_2)$, the spaces of holomorphic functions,
		are isomorphic as $\mc$-algebras with unit.
		This implies that the holomorphic structure of a pseudoconvex domain
		is uniquely determined by the algebraic structure of the space of holomorphic functions on the domain.
		
		In the above result, the multiplicative  structure on the spaces of holomorphic functions plays an essential role.
		In the present work, we prove some results in a related but different direction.
		
		Let $ D$ be a domain in $\mathbb C^n$.
		Let $z=(z_1,\cdots,z_n)$ be the natural holomorphic coordinates of $\mathbb C^n$,
		and let $d\lambda_n:={(\frac{i}{2})^n}dz_1\wedge d\overline{z}_1\wedge\cdots\wedge dz_n\wedge d\overline{z}_n$
		be the canonical volume form on $ D$.
		For $p>0$, we denote by  $A^p( D)$  the space of all holomorphic functions $\phi$ on $ D$ with finite $L^p$-norm
		\begin{align*}
			\|\phi\|_p:=\left(\int_ D |\phi|^pd\lambda_n\right)^{1/p}.
		\end{align*}
		It is a standard fact that for $p\geq 1$,  $A^p( D)$ are
		separable Banach spaces, and for $0<p<1$, $A^p( D)$ are complete
		separable metric spaces with respect to the metric
		$$d(\varphi_1,\varphi_2):=\|\varphi_1-\varphi_2\|_p^p.$$
		
		The $p$-Bergman kernel is defined as follows:
		\begin{align*}
			B_{ D,p}(z):=\sup_{\phi\in A^p( D)}\frac{|\phi(z)|^2}{\|\phi\|^2_p}.
		\end{align*}
		When $p=2$, $B_{ D,p}$ is the ordinary Bergman kernel.
		By a standard argument of the  Montel theorem,
		one can prove that  $B_{ D,p}$ is a continuous plurisubharmonic function on $ D$.
		We say that $B_{ D,p}$ is exhaustive if for any real number $c$ the set $\{z\in D| B_{ D,p}(z)\leq c\}$ is compact.
		
		A bounded domain $ D$ is called hyperconvex if there is a
		plurisubharmonic function $\rho: D\ra [-\infty, 0)$ such that
		for any $c<0$ the set $\{z\in  D|\rho(z)\leq c\}$ is compact.

		In this paper, in stead of $\mathcal O(D)$, we will consider the space $A^p( D)$ (for some fixed $p>0$) as linear invariants of bounded domains.
		The main purpose is to prove, under certain conditions, that two bounded domains $D_1$ and $D_2$ are holomorphically equivalent if there exists
		a linear isometry   $T:A^p( D_1)\ra A^p( D_2) $ between $A^p(D_1)$ and $A^p(D_2)$, i.e., $T$ is a linear isomorphism and $\|T(\phi)\|_p=\|\phi\|_p$ for all $\phi\in A^p( D_1)$.
		Note that $A^p( D)$ carries a norm structure (a geometric structure) but there is no natural multiplicative structure on it.
		In some sense, the norm structure on $A^p( D)$ plays similar role of the multiplicative structure on $\mathcal O(D)$ in our consideration.

		It is proved in \cite{DWZZ20} the following
		\begin{thm}\label{thm-intro: main theorem in previous paper}
			Let $ D_1\subset\mc^n$ and $ D_2\subset\mc^m$ be bounded domains.
			If $D_1$ and $D_2$ are hyperconvex, and
			suppose that there is a $p>0, p\neq 2,4,6 \cdots$ , such that
			\begin{itemize}
				\item[(1)] there is a  linear isometry $T:A^p( D_1)\ra A^p( D_2)$,  and
				\item[(2)] the $p$-Bergman kernels of  $ D_1$ and $ D_2$ are exhaustive,
			\end{itemize}
			then $m=n$ and there exists a unique biholomorphic map $F: D_1\ra  D_2$ such that
			\begin{align*}
				|T(\phi)\circ F||J_F|^{2/p}=|\phi|, \ \forall \phi\in A^p( D_1),
			\end{align*}
			where $J_F$ is the holomorphic Jacobian of $F$.
			If $n=1$, the assumption of hyperconvexity can be dropped.
		\end{thm}

		For the case that $n=m=1$ and $p=1$,
		Theorem \ref{thm-intro: main theorem in previous paper} was proved by Lakic \cite{Lakic97} and Markovic in \cite{Mark03}.
		Application of spaces of pluricanonical forms with pseudonorms to birational geometry of projective algebraic
		manifolds  was proposed by Chi and Yau in \cite{Chi-Yau2008} and was studied in \cite{An, Chi2016, Yau15}.

		A relative version of Theorem \ref{thm-intro: main theorem in previous paper} is established by Inayama in \cite{Ina19}.

		The purpose of the present paper is to generalize Theorem \ref{thm-intro: main theorem in previous paper} to more general domains.

		
		The first observation is to replace the $p$-Bergman kernel exhaustion condition in Theorem \ref{thm-intro: main theorem in previous paper} by
		a  new condition called $A^p$-completeness.
		
		\begin{defn}\label{def-intro:A^p complete}
			A bounded domain $D\subset\mc^n$ is \emph{$A^p$-complete} if there does not exist a domain $\tilde D$
			with $D\subsetneqq \tilde D$ such that the restriction map
			$i:A^p(\tilde{D})\rightarrow A^p(D)$ is a linear isometry.
		\end{defn}
		
		It is obvious that, if the $p$-Bergman kernel of $D$ is exhaustive, then $D$ is $A^p$-complete. But the converse is not true in general. In fact, $D$ is  $A^p$-complete for any $p>0$ if $\mathring{\overline{D}}= D$, namely, the interior of the closure of $D$ is $D$ itself. Note that $D$ must be pseudoconvex if the $p$-Bergman kernel on $D$ is exhaustive, so   the condition of  $A^p$-completeness is much weaker than the $p$-Bergman exhaustion condition.

		
		It is proved in \cite{NZZ16} that for $0<p<2$ the $p$-Bergman kernel on any bounded pseudoconvex domain is exhaustive,
		so such domains are $A^p$-complete for all $p\in (0,2)$.
		
		In our discussion, the importance of the property of $A^p$-completeness is encoded in the following result.
		For $A^p$-complete domains $D_1$ and $D_2$,  if there is a linear isometry from $A^p(D_1)$ to $A^p(D_2)$,
		then $D_1$ and $D_2$ are almost holomorphically equivalent, in the sense that there exist hypersurfaces $A_1$ and $A_2$ in $D_1$ and $D_2$ respectively, such that $D_1\backslash A_1$ and $D_2\backslash A_2$ are holomorphically equivalent.
		(see Theorem \ref{thm:biholo surface comp.} for details).
		
		Our  second observation is to replace  the hyperconvexity condition in Theorem \ref{thm-intro: main theorem in previous paper} by
		a weaker new condition called being not of boundary blow down type (BBDT for short, see \S \ref{sec:bbdt} for definition).
		
		The main result of the present paper  is
		
		\begin{thm}\label{thm-intro:not bbdt implies biholomorphic}
			Assume that $D_1$ and $D_2$ are $A^p$-complete domains which are not of boundary blow down type, and $T:A^p(D_1)\ra A^p(D_2)$ is a linear isometry,
			for some $p>0$ and $p$ is not an even integer.
			Then there exists a unique biholomorphic map $F:D_1\ra D_2$ such that $J_F(z)^{2/p}$ has a single-valued branch on $D_1$,
			and
			$$T\phi(F(z))J_F(z)^{2/p}=\lambda\phi(z),\ \forall \phi\in A^p(D_1), z\in D_1,$$
			where  $\lambda$ is a constant with  $|\lambda|=1$.
		\end{thm}
		
		
		For $p\in (0,2)$, we believe that the condition of not being BBDT in the above theorem  is not necessary,  namely, we have the following
		
		\begin{conjecture}\label{conj:A^p complete implies iso}
			Let $D_1$ and $D_2$ be $A^p$-complete bounded domains for some $p\in (0,2)$.
			If there is a linear isometry between $A^p(D_1)$ and $A^p(D_2)$, then $D_1$ and $D_2$ are holomorphically equivalent.
		\end{conjecture}
		
		In the last section, we will construct an example to show that this conjecture is not true for $p>2$.

		Here we  point out that a domain $D$ is not of BBDT if $D$ can be represented as
		$\tilde D\backslash K$ for some hyperconvex domain $\tilde D$ and compact subset $K$ (may be empty) of $\tilde D$.
		Note that $D$ can not be pseudoconvex if $K\neq \emptyset$.
		So one of the novelties of Theorem \ref{thm-intro:not bbdt implies biholomorphic} is that we can handle some domains
		that are not pseudoconvex, which is a point that beyond the scope in \cite{DWZZ20}.
		
		It is known that bounded pseudoconvex domains with H\"older boundary are hyperconvex \cite{Che21}.
		Of course such a domain $D$ must satisfy $\mathring{\overline{D}}=D$ and hence is $A^p$-complete.
		So we get a corollary of Theorem \ref{thm-intro:not bbdt implies biholomorphic} as follows,
		which can not be deduced from Theorem \ref{thm-intro: main theorem in previous paper} if $p>2$, since we do not know if the $p$-Bergman kernels of $D_1$ and  $D_2$ are exhaustive.
		
		\begin{cor}\label{cor:Libschitz boundary case}
			Assume that $D_1$ and $D_2$ are bounded pseudoconvex domains in $\mc^n$ with H\"older boundary.
			If there exists a linear isometry between $A^p(D_1)$ and $A^p(D_2)$
			for some $p>0$ and $p$ is not an even integer,
			then $D_1$ and $D_2$ are holomorphically equivalent.
		\end{cor}
		
		%
		
		For further study, it is possible to generalize the methods and results in the present paper to
		complex manifolds equipped with Hermitian holomorphic vector bundles and develop some relative version of them.

		\subsection*{Acknowledgements}
		The authors thank Zhenqian Li for helpful discussions on related topics.
		This research is supported by National Key R\&D Program of China (No. 2021YFA1002600). The authors are partially supported respectively by NSFC grants (11871451,  11801572, 12071035, 11688101).
		The first author is  partially supported by the Fundamental Research Funds for the Central Universities.
		The third author is partially supported by Beijing Natural Science Foundation (1202012, Z190003).

		\section{A measure theoretic preparation}\label{sec:Rudin and measure theory}
		As in the works  in \cite{Mark03,DWZZ20}, one of the key tool for our discussion in the present paper is the following result of Rudin.
		
		\begin{lem}[\cite{Rud76}]\label{lem:equ lp}
			Let $\mu$ and $\nu$ be finite positive measures on two sets $M$ and $N$ respectively.
			Assume $0<p<\infty$ and $p$ is not even.
			Let $n$ be a positive integer.
			If $f_i\in  L^p(M,\mu)$, $g_i\in L^p(N,\nu)$ for $1\leq i\leq n$ satisfy
			\begin{align*}
				\int_M|1+\alpha_1f_1+\cdots +\alpha_nf_n|^pd\mu=\int_N|1+\alpha_1g_1+\cdots +\alpha_ng_n|^pd\nu
			\end{align*}
			for all $(\alpha_1,\cdots,\alpha_n)\in \mathbb{C}^n$,
			then $(f_1,\cdots, f_n)$ and $(g_1,\cdots, g_n)$ are equimeasurable,
			i.e. for every bounded Borel measurable function (and for every real-valued nonnegative Borel function) $u:\mathbb C^n\ra \mathbb C$, we have
			\begin{align*}
				\int u(f_1,\cdots, f_n)d\mu=\int u(g_1,\cdots, g_n)d\nu.
			\end{align*}
			Furthermore, let $I:M\ra \mathbb C^n$ and $J: N\ra \mathbb C^n$ be the maps $I=(f_1,\cdots, f_n)$ and $J=(g_1,\cdots, g_n)$, respectively. Then we have
			\begin{align*}
				\mu(I^{-1}(E))=\nu(J^{-1}(E))
			\end{align*}
			for every Borel set $E$ in $\mathbb C^n$.
		\end{lem}
		
		The following lemma is a direct corollary of Lemma \ref{lem:equ lp}.
		
		\begin{lem}[{\cite[Lemma 2.2]{DWZZ20}}]\label{lem: equi lp higher dimension}
			Let $ D_1$ and $ D_2$ be two bounded  domains in $\mathbb C^n$ and $\mathbb C^m$, repectively.
			Suppose that $\phi_k$, $k=0,1,2,\cdots,N$, $N\in \mathbb N$, are elements of $A^p( D_1)$ and
			suppose that $\psi_k$, $k=0,1,2,\cdots, N$ are elements of $A^p( D_2)$,
			such that for every $N$-tuple of complex numbers $\alpha_k$, $k=1,\cdots, N$,
			we have
			\begin{align*}
				\|\phi_0+\sum\limits_{k=1}^N\alpha_k\phi_k\|_p=\|\psi_0+\sum\limits_{k=1}^N\alpha_k\psi_k\|_p.
			\end{align*}
			If neither $\phi_0$ nor $\psi_0$ is constantly zero,
			then for every real valued non-negative Borel function $u:\mathbb C^N\ra \mathbb C$,
			we have
			\begin{align*}
				\int_{ D_1} u(\frac{\phi_1}{\phi_0},\cdots, \frac{\phi_N}{\phi_0})|\phi_0|^pd\lambda_n=\int_{ D_2} u(\frac{\psi_1}{\psi_0},\cdots, \frac{\phi_N}{\phi_0})|\psi_0|^pd\lambda_m.
			\end{align*}
		\end{lem}
		
		
		\section{The holomorphic map associated to a linear isometry}\label{sec:construction of hol map }
		Let $D$ be a bounded domain in $\mathbb C^n$.
		For $p>0$, as in the introduction, we denote by  $A^p( D)$  the space of all holomorphic functions $\phi$ on $D$ with finite $L^p$-norm
		\begin{align*}
			\|\phi\|_p:=\left(\int_ D |\phi|^pd\lambda_n\right)^{1/p}.
		\end{align*}
		
		For two given bounded domains $D_1\subset\mc^n$ and $D_2\subset\mc^m$ and a linear isometry $T:A^p(D_1)\ra A^p(D_2)$ for some $p>0$,
		we try to construct an associated biholomorphic map $F:D_1\ra D_2$.
		
		By definition, a hyperplane in $A^p(D)$ is the kernel of a nonzero continuous linear functional on $A^p(D)$.
		For $z\in D$, the evaluation map
		$$A^p(D)\ra\mc;\ \phi\mapsto \phi(z)$$
		is a continuous linear functional on $A^p(D)$.
		So a natural way to connect points in $D$ and $A^p(D)$ is that:
		any point $z\in D$ corresponds to a hyperplane
		$$H_{D,z}:=\{\phi\in A^p(D)|\phi(z)=0\}$$
		of $A^p(D)$.
		Since $D$ is bounded, $A^p(D)$ separates points on $D$ and hence $z$ is uniquely determined by $H_{D,z}$.
		Let $\mathbb P(A^p(D))$ be the set of all hyperplanes in $A^p(D)$,
		we then get an injective map
		$$\sigma_D:D\ra \mathbb P(A^p(D)),\ z\mapsto H_{D,z}.$$
		
		From the linear isometry $T:A^p(D_1)\ra A^p(D_2)$, we get a natural bijective map
		$$T_*: \mathbb P(A^p(D_1))\ra \mathbb P(A^p(D_2)),\ H\mapsto T(H).$$
		We define subsets $D'_1\subset D_1$ and $D'_2\subset D_2$ as follows: for $z\in D_1, w\in D_2$,
		by definition, $z\in D'_1$ and $w\in D'_2$ if and only if $T(H_{D_1,z})=H_{D_2,w}$, namely,
		$T$ maps the hyperplane in $A^p(D_1)$ associated to $z$ to the hyperplane in $A^p(D_2)$ associated to $w$.
		Then we can define a bijective map $F:D'_1\ra D'_2$ by setting $F(z)=w$ provided $T(H_{D_1,z})=H_{D_2,w}$.
		In other words, $D'_1=\sigma^{-1}_{D_1}\circ T^{-1}_*\circ\sigma_{D_2}(D_2)$, $D'_2=\sigma^{-1}_{D_2}\circ T_*\circ \sigma_{D_1}(D_1)$,
		and $F=\sigma^{-1}_{D_2}\circ T_*\circ \sigma_{D_1}$.
		
		Other equivalent formulations of $D'_1$, $D'_2$ and $F$ are as follows.
		\begin{lem}\label{lem: F in another form}
			For $z\in D_1, w\in D_2$, the following conditions are equivalent:
			\bi
			\item[(1)] $z\in D'_1$, $w\in D'_2$, and $F(z)=w$,
			\item[(2)] $\phi_1(z)T\phi_2(w)=T\phi_1(w)\phi_2(z)$ for all $\phi_1,\phi_2\in A^p(D_1)$,
			\item[(3)] there exists $\lambda\in\mc^*$ such that $T\phi(w)=\lambda\phi(z)$ for all $\phi\in A^p(D_1)$.
			\ei
		\end{lem}
		\begin{proof}
			Let $l_z:A^p(D_1)\ra\mc$ and $l_w:A^p(D_2)\ra\mc$ be the evaluation maps.
			Then $T_*(H_{D_1,z})=H_{D_2,w}$ if and only if $l_w(T\phi)=\lambda l_z(\phi)$ for some $\lambda\in\mc^*$.
			So (1) implies (3).
			Other implications are obvious.
		\end{proof}
		
		The following lemma is a direct consequence of Lemma \ref{lem: F in another form}, which implies that $F$ is a closed map.
		
		\begin{lem}\label{lem:F closed map}\
			\bi
			\item[(1)] Let $z_j$ be a sequence in $D'_1$ with $z_j\ra z\in D_1$ as $j\ra \infty$.
			If $w_j:=F(z_j)$ converges to some $w\in D_2$, then $z\in D'_1, w\in D'_2$ and $F(z)=w$.
			\item[(2)] For $z\in D'_1$ and $\phi\in A^p(D_1)$, $\phi(z)=0$ if and only if $T\phi(F(z))=0$.
			\ei
		\end{lem}
		
		To study $D'_1, D'_2$ and $F$ concretely, following the idea in \cite{Mark03}, we take an arbitrary countable dense
		subset of $A^p(D_1)$ and express $F$ in term of it, as follows.
		
		Let $\{\phi_j\}^{+\infty}_{j=0}$ be a countable dense subset of $A^p(D_1)$ and let $\psi_j=T\phi_j$,
		then $\{\psi_j\}^{+\infty}_{j=0}$ is a countable dense subset of $A^p(D_2)$.
		We assume that $\phi_0$ is not identically 0.
		
		Roughly speaking, we may view $[\phi_0(z):\phi_1(z):\cdots]$ as the homogenous coordinate on $\mathbb P(A^p(D_1))$,
		and view $\left(\frac{\phi_1(z)}{\phi_0(z)}, \frac{\phi_2(z)}{\phi_0(z)},\cdots \right)$ as the inhomogenous  coordinate on $\mathbb P(A^p(D_1))$.
		
		For $N\in \mathbb N_+$, we define maps $I_N, J_N$ as follows:
		$$I_N: D_1\ra\mc^N,\ z\mapsto \left(\frac{\phi_1(z)}{\phi_0(z)},\cdots, \frac{\phi_N(z)}{\phi_0(z)}\right),$$
		$$J_N: D_2\ra\mc^N,\ w\mapsto \left(\frac{\psi_1(w)}{\psi_0(w)},\cdots, \frac{\psi_N(w)}{\psi_0(w)}\right),$$
		which are measurable maps on $ D_1$ and $ D_2$ respectively.
		We can also define $I_\infty$ and $J_\infty$ as
		$$I_\infty: D_1\ra\mc^\infty,\ z\mapsto \left(\frac{\phi_1(z)}{\phi_0(z)}, \frac{\phi_2(z)}{\phi_0(z)}, \cdots\right),$$
		$$J_\infty: D_2\ra\mc^\infty,\ w\mapsto \left(\frac{\psi_1(w)}{\psi_0(w)}, \frac{\psi_2(w)}{\psi_0(w)}, \cdots\right).$$
		Indeed, $I_N, I_\infty$ are defined on $D_1\backslash \phi^{-1}_0(0)$, and $J_N, J_\infty$ are defined on $D_2\backslash \psi^{-1}_0(0)$. Since $A^p(D_i)$ seperates points in $D_i$, $i=1,2$, the maps $I_\infty$ and $J_\infty$ are injective.
		
		By Lemma \ref{lem: F in another form}, for $z\in D_1\backslash \phi^{-1}_0(0)$ and $w\in D_2\backslash \psi^{-1}_0(0)$,
		we have $z\in D'_1, w\in D'_2$ and $F(z)=w$ if and only if $I_\infty(z)=J_\infty(w)$.
		For $z\in D'_1\backslash \phi^{-1}_0(0)$, we have
		$$F(z)=J^{-1}_\infty(I_\infty(z)).$$
		It is clear that
		$$D'_1\backslash\phi^{-1}_0(0)=\cap_N I^{-1}_N J_N( D_2\setminus\psi_0^{-1}(0))=I_\infty^{-1}J_\infty( D_2\setminus\psi_0^{-1}(0)),$$
		$$ D'_2\backslash\psi^{-1}_0(0)=\cap_N J_N^{-1}I_N( D_1\setminus\phi_0^{-1}(0))=J_\infty^{-1}I_\infty( D_1\setminus\phi_0^{-1}(0)).$$

		We now assume that $p>0$ is a real number and $p$ is not an even integer.
		The following results are proved in \S 2.2 in \cite{DWZZ20}.
		
		\begin{lem}\label{lem: D' full measure}
			The Lebesgue measure of $ D_j\backslash D'_j$ are zero, for $j=1, 2$.
		\end{lem}

		\begin{lem}\label{lem:equi dimension}
			The dimensions of $ D_1$ and $ D_2$ are equal, namely $n=m$.
		\end{lem}

		\begin{lem}\label{lem: D' open}
			$ D'_j$ are open subsets of $ D_j$, for $j=1, 2$.
		\end{lem}

		\begin{lem}\label{lem:T1=Jacobi}
			For $z\in D'_1$ and $\phi\in A^p( D_1)$, we have
			$$|T(\phi)(F(z))||J_F(z)|^{2/p}=|\phi(z)|, \ \phi\in A^p( D_1), z\in D'_1,$$
			where $J_F$ is the holomorphic Jacobian of $F$.
		\end{lem}
		
		\section{$A^p$-completeness of bounded domains}\label{sec:A^p complete}
		We first introduce a new notion, namely, $A^p$-completeness of bounded domains.
		
		\begin{defn}\label{def:A^p complete}
			Let $D\subset\mc^n$ be a bounded domain and $p>0$.
			We say that $D$ is \emph{$A^p$-complete} if there does not exist a domain $\tilde D$
			with $D\subsetneqq \tilde D$ such that the restriction map
			$i:A^p(\tilde{D})\rightarrow A^p(D)$ is an isometry.
		\end{defn}
		
		The following two basic facts related to  $A^p$-completeness is obvious:
		\bi
		\item[(1)] $D$ must be $A^p$-complete for any $p>0$ if $\mathring{\overline{D}}=D$, namely, the interior of the closure of $D$ is $D$ itself;
		\item[(2)] if the $p$-Bergman kernel $B_{D,p}$ of $D$ is exhaustive, then $D$ is $A^p$-complete.
		\ei
		
		It is proved in \cite{NZZ16} that for $p<2$ the $p$-Bergman kernel on any bounded pseudoconvex domain is exhaustive,
		it follows that all bounded pseudoconvex domains in $\mc^n$ are $A^p$-complete for any $0<p<2$.

		Assume that $D_1$ and $D_2$ are bounded domains in $\mc^n$, and $T:A^p(D_1)\ra A^p(D_2)$ is a linear isometry.
		We now take a dense countable subset $\{\phi_j\}_{j\geq 0}$ of $A^p(D_1)$,
		such that $\psi_0=1,\psi_1=w_1, \cdots, \psi_n=w_n$, and $(w_1,\cdots, w_n)$ are the natural linear coordinates on $D_2$.
		We denote the zero set $\phi_0^{-1}(0)$ of $\phi_0=T^{-1}(1)\in A^p(D_1)$ by $A_1$.
		Then it is clear that $F=I_n=\left(\frac{\phi_1(z)}{\phi_0(z)},\cdots, \frac{\phi_n(z)}{\phi_0(z)}\right)$ on $D'_1\backslash A_1$.
		
		\begin{prop}\label{prop:domain of F}
			Assume that $D_1$ and $D_2$ are bounded domains in $\mc^n$, and $T:A^p(D_1)\ra A^p(D_2)$ is a linear isometry.
			If $D_2$ is $A^p$-complete,
			then $D'_1=D_1\backslash A_1$.
		\end{prop}
		\begin{proof}
			We divide the proof into three steps.
			
			\emph{Step 1.} Prove that $I_n(D_1\backslash A_1)\subset D_2$.
			
			By Lemma \ref{lem: F in another form}, we have
			$$\psi(F(z))=\frac{\phi(z)}{\phi_0(z)}, \ z\in D'_1\backslash A_1,$$
			for all $\phi\in A^p(D_1)$ and $\psi\in T\phi\in A^p(D_2)$.
			Since $D_1'$ is dense in $D_1$ by Lemma \ref{lem: D' full measure} and $I_n=F$ on $D'_1\backslash A_1$,
			by Lemma \ref{lem:T1=Jacobi}, we get
			\begin{align}\label{equ: Jac}|J_{I_n}(z)|^{2/p}=|\phi_0(z)|,\ z\in D_1\backslash A_1.\end{align}
			In particular, $J_{I_n}(z)\neq 0$ for $z\in D_1\backslash A_1$.
			
			
			We now argue by contradiction.
			Assume that there exists $z_0\in D_1\backslash A_1$ such that $w_0:=I_n(z_0)\notin D_2$.
			
			We have seen that $J_{I_n}(z_0)\neq 0$.
			So there exists a neighborhood $U$ of $z_0$ in $D_1\backslash A_1$ and a neighborhood $V$ of $w_0$ in $\mc^n$ such that
			$I_n(U)=V$ and $I_n|_U:U\ra V$ is a biholomorphic map.
			
			Let $\tilde D=D_2\cup V$. Note that $I_n$ is smooth, $D'_1\backslash A_1$ has full measure in $D_1$ (by Lemma \ref{lem: D' full measure}),
			and $I_n(D'_1\backslash A_1)\subset D_2$,
			it follows that $D_2\varsubsetneqq \tilde D\subset \overline{D_2}$ and $\tilde D\backslash D_2$ has null measure, where $\overline{D_2}$ is the closure of $D_2$ in $\mc^n$.
			
			Let $\psi\in A^p(D_2)$ and $\phi\in A^p(D_1)$ with $\psi=T\phi$.
			We now continue $\psi$ holomorphically to a holomorphic function $\tilde \psi$ on $\tilde D$ as follows.
			Let $\psi'\in\mathcal O(V)$ be given as
			$$\psi'=\frac{\phi}{\phi_0}\circ (I_n|_U)^{-1}.$$
			Then $\psi=\psi'$ on the dense subset $F(D'_1\backslash A_1)\cap D_2\cap V$ of $D_2\cap V$,
			it follows that $\psi=\psi'$ on $D_2\cap V$.
			Thus we can glue $\psi$ and $\psi'$ together to ge a holomorphic function $\tilde\psi$ on $\tilde D$,
			which is an extension of $\psi$ by definition.
			Since $\tilde D\backslash D_2$ has null measure.
			It follows that
			$$\int_{\tilde D}|\tilde\psi|^p=\int_{D_2}|\psi|^p.$$
			Note that $\psi\in A^p(D_2)$ is arbitrary, $D_2$  can not be $A^p$-complete, which is a contradiction.
			
			\emph{Step 2.} Prove that $D_1\backslash A_1\subset D_1'$.
			Let $z\in D_1\backslash A_1$. Since $D'_1$ is dense in $D_1$,
			there exists a sequence $\{z_j\}_{j\geq 1}$ in $D'_1\backslash A_1$ that converges to $z$.
			Since $F(z_j)=I_n(z_j)$ for all $j$.
			By the conclusion in Step 1 and the continuity of $I_n$, we get $F(z_j)\ra I_n(z)\in D_2$ as $j\ra\infty$.
			It follows from Lemma \ref{lem:F closed map} that $z\in D_1'$.
			
			\emph{Step 3.} Prove that $D'_1=D_1\backslash A_1$.
			
			From Step 2, and  Riemann's removable singularity theorem, $F$ extends to a holomorphic map, denoted by $F_1$, from $D_1$ to $\mc^n$,
			with $F(D_1)\subset\overline{D_2}$.
			By (\ref{equ: Jac}) and continuity, we have
			$$|J_{F_1}(z)|^{2/p}=|\phi_0(z)|$$
			on $D_1$. It follows that $J_{F_1}(z)=0$ on $A_1$.
			From that we deduce that $A_1\cap D'_1=\emptyset$,
			and it follows that $D'_1=D_1\backslash A_1$.
		\end{proof}

		As in Proposition \ref{prop:domain of F},
		we assume that $D_1$ and $D_2$ are bounded domains in $\mc^n$, and $T:A^p(D_1)\ra A^p(D_2)$ is a linear isometry.
		We assume in addition that both $D_1$ and $D_2$ are $A^p$-complete.
		
		Let $A_1=(T^{-1}(1))^{-1}(0)=\phi^{-1}_0(0)$ and let $A_2=(T(1))^{-1}(0)$.

		Exchanging the roles of $D_1$ and $D_2$, we can define a holomorphic map $F_2:D_2\ra\mc^n$ with $F_2(D_2)\subset \overline{D_1}$.
		For the same reasons we have
		$$|J_{F_2}(w)|^{2/p}=|T(1)(w)|$$
		on $D_2$, and $D'_2=D_2\backslash A_2$.
		
		By Lemma \ref{lem:F closed map}, we get
		$$F_1(A_1)\subset \partial D_2,\ F_2(A_2)\subset\partial D_1.$$
		
		We summarize the above results as the following theorem.
		\begin{thm}\label{thm:biholo surface comp.}
			Assume that $D_1$ and $D_2$ are $A^p$-complete domains in $\mc^n$, and $T:A^p(D_1)\ra A^p(D_2)$ is a linear isometry.
			Then:
			\bi
			\item[(1)] $D'_1=D_1\backslash A_1$ and $D'_2=D_2\backslash A_2$,
			\item[(2)] $F:D_1\backslash A_1\ra D_2\backslash A_2$ is a biholomorphic map,
			\item[(3)] $|J_{F_1}(z)|^{2/p}=|\phi_0(z)|$ on $D_1$, and $|J_{F_2}(w)|^{2/p}=|T(1)(w)|$ on $D_2$,
			\item[(4)] $F_1(A_1)\subset \partial D_2$ and $F_2(A_2)\subset \partial D_1$.
			\ei
		\end{thm}
		
		
		In the next section, we will prove that, under certain additional conditions,  $A_1$ and $A_2$ are empty.
		
		\section{Domains of boundary blow down type (BBDT) }\label{sec:bbdt}
		We first give an observation to the picture presented in Theorem \ref{thm:biholo surface comp.}.
		We preserve the notations, definitions, and assumptions as in Theorem \ref{thm:biholo surface comp.}.
		
		Gluing $D_1$ and $D_2$ by identifying $z\in D_1\backslash A_1$ and $w=F(z)\in D_2\backslash A_2$,
		we get a complex manifold, say $D:=D_1\sqcup_F D_2$.
		For the proof of this statement, we only need to check that $D$ is Hausdorff.
		It suffices to show that any $z_1\in A_1$ and $z_2\in A_2$
		have disjoint neighborhoods in $D$. The natural inclusions
		$j_1:D_1\rightarrow D$ and $j_2:D_2\rightarrow D$ are open maps.
		Since $F_1(z_1)\in \partial D_2$ and $F_2(z_2)\in\partial D_1$,
		there exist neighborhoods $U_k$ of $z_k$ in $D_k$ for $k=1,2$,
		such that $F_1(U_1)\cap U_2=\emptyset$ and $F_2(U_2)\cap U_1=\emptyset$,
		then $j_1(U_1)\cap j_2(U_2)=\emptyset$.
		We can view $A_1$ and $A_2$ as hypersurfaces of $D$ which are disjoint.
		
		We have $J_{F_1}\in\mathcal{O}(D_1)$ and $J_{F_2}\in\mathcal{O}(D_2)$,
		and $J_{F_1}(z)=\frac{1}{J_{F_{2}}(w)}$ for $z\in D_1\backslash A_1$ with $w=F_1(z)$.
		So we can define a meromorphic function $J$ on $D$ as follows:
		$$
		\begin{cases}J(j_1(z))=J_{F_1}(z),\quad\quad  z\in D_1\\
			J(j_2(w))= \frac{1}{J_{F_{2}}(w)}, \quad   w\in D_2\backslash A_2.
		\end{cases}
		$$
		Then $J\in\mathcal{M}(D)\cap\mathcal{O}(D\setminus A_2)$ and $J^{-1}(\infty)=A_2$,
		where $\mathcal M(D)$ denotes the space of meromorphic functions on $D$.
		
		We extract an abstract concept from this picture as follows.
		\begin{defn}\label{def:BBDP}
			A domain $D\subset\mathbb{C}^n$ is of \emph{boundary blow down type} (BBDT for short),
			if there exists a complex manifold $M$, a (nonempty) hypersurface $A\subset M$,
			$h\in\mathcal{M}(M)\cap\mathcal{O}(M\setminus A)$, and a holomorphic map $\sigma:M\rightarrow \mc^n$,
			such that
			\bi
			\item[(i)] $\sigma(M\backslash A)=D$ and  $\sigma(A)\subset\partial D$,
			\item[(ii)] $\sigma|_{M\backslash A}: M\backslash A\ra D$ is a biholomorphic map,
			\item[(iii)] $h^{-1}(\infty)=A$.
			\ei
		\end{defn}
		
		Go back to our previous construction.
		If $A_2\neq \emptyset$, then $D_1$ is of BBDT.
		
		It is not easy to construct a domain of BBDT.
		We will see an example in the final section.
		
		
		The following notion is also tightly related to our discussion.
		\begin{defn}
			A complex manifold $X$ is said to have \textit{{punctured disc property}} (PDP for short)  if
			there is a proper holomorphic map from $\mathbb D^*$ to $X$,
			where $\mathbb D^*=\{z\in\mc| 0<|z|\leq 1\}$.
		\end{defn}

		
		
		Recall that a domain $D\subset\mc^n$ is called a Runge domain if $D$ is holomorphically convex and any holomorphic function
		on $D$ can be approximated uniformly on compact subsets of $D$ by polynomials.
		\begin{lem}\label{lem:domain have PDP}
			A bounded domain $D$ does not have  PDP if either $D$ is Runge or $D$ is hyperconvex.
		\end{lem}
		\begin{proof}
			We argue by contradiction.
			If $D$  has  PDP, then there is a proper holomorphic map $h:\mathbb D^*\ra D$.
			
			If $D$ is Runge, then the polynomial convex hull $\widehat{h(S^1)}$ of $h(S^1)$ lies in $D$.
			On the other hand, by Riemann's removable singularity theorem and the maximum principle for holomorphic functions,
			we see that $h(\mathbb D^*)\subset \widehat{h(S^1)}$, which contradicts to the assumption that $h:\mathbb D^*\ra D$ is proper.
			
			If $D$ is hyperconvex, then there is a strictly p.s.h function $\rho:D\ra [-\infty, 0)$ such that for any $c<0$
			the set $\{z\in D| \rho(z)\leq c\}$ is compact.
			Consider $\tilde\rho:=\rho\circ h$, which is a  subharmonic function on $\mathbb D^*$. Note that $\tilde\rho$  is bounded above, it can be  extended to a  subharmonic function on $\mathbb D$,
			which is also denoted by $\tilde\rho$.
			It is clear that $\tilde\rho$ attains its maximum 0 at the origin,
			hence $\tilde\rho$ is a constant function,
			which contradicts to the assumption that $\rho$ is strictly plurisubharmonic.
		\end{proof}
		
		\begin{thm}\label{thm:pdp implies not bbdt}
			Let $D\subset\mathbb{C}^n\ (n>1)$ be a bounded domain that does not have PDP, then
			\bi
			\item[(i)] $D$ is not of BBDT,
			\item[(ii)] $D\backslash K$ is not of BBDT for any compact set $K\subset D$ such that $D\backslash K$ is connected.
			\ei
		\end{thm}
		\begin{proof}
			The proof of (i) is trivial by definition.
			We prove (ii) using argument by contradiction.
			Assume $\tilde{D}:=D\backslash K$ is of boundary blow down type.
			By definition,  there exists a complex manifold $M$, a (nonempty) hypersurface $A\subset M$,
			$h\in\mathcal{M}(M)\cap\mathcal{O}(M\setminus A)$, and a holomorphic map $\sigma:M\rightarrow \mc^n$,
			such that
			\bi
			\item[(i)] $\sigma(M\backslash A)=\tilde D$ and  $\sigma(A)\subset\partial\tilde D$,
			\item[(ii)] $\sigma|_{M\backslash A}: M\backslash A\ra \tilde D$ is a biholomorphic map,
			\item[(iii)] $h^{-1}(\infty)=A$.
			\ei
			Note that $\partial \tilde{D}\subset \partial D\cup K$. Since $D$ does not have PDP, we have $\sigma(A)\subset K$.
			Let $h'=h\circ (\sigma|_{M\backslash A})^{-1}$, then $h'\in\mathcal{O}(\tilde{D})$.
			We can extend $h'$ holomorphically to $D$ by Hartogs' theorem.
			This contradicts to the fact that $h^{-1}(\infty)=A$.
		\end{proof}

		
		We go back to the setting of Theorem \ref{thm:biholo surface comp.},
		by the above discussion, we see that $A_1$ and $A_2$ are both empty sets and get a biholomorphic map $F:D_1\ra D_2$, if $D_1$ and $D_2$ are not of BBDT.
		Combing this with Lemma \ref{lem:T1=Jacobi}, we get the following
		
		\begin{thm}[=Theorem \ref{thm-intro:not bbdt implies biholomorphic}]\label{thm:not bbdt implies biholomorphic}
			Assume that $D_1$ and $D_2$ are $A^p$-complete domains which are not of boundary blow down type, and $T:A^p(D_1)\ra A^p(D_2)$ is a linear isometry,
			for some $p>0$ and $p$ is not an even integer.
			Then there exists a unique biholomorphic map $F:D_1\ra D_2$ such that $J_F(z)^{2/p}$ has a single-valued branch on $D_1$,
			and
			$$T\phi(F(z))J_F(z)^{2/p}=\lambda\phi(z),\ \forall \phi\in A^p(D_1), z\in D_1,$$
			where  $\lambda$ is a constant with  $|\lambda|=1$.
		\end{thm}
		
		\begin{proof}
			By the above discussion and Lemma \ref{lem:T1=Jacobi}, we get a biholomorphic map $F:D_1\ra D_2$ such that
			$$|T\phi(F(z))||J_F(z)|^{2/p}=|\phi(z)|,\ \forall \phi\in A^p(D_1), z\in D_1.$$
			Considering $\phi_0\in A^p(D_1)$ with $T\phi_0=1\in A^p(D_2)$,
			we see that
			$$|J_F(z)|^{2/p}=|\phi_0(z)|, \ z\in D_1.$$
			Let $B$ be a ball in $D_1$, then there exists a single-valued branch of $J_F^{2/p}$ on $B$ which equals to $\lambda\phi_0$ for some
			$\lambda\in\mc$ with $|\lambda|=1$. Note that $D_1$ is connected and $\phi_0$ is globally defined on $D_1$,
			$\lambda\phi_0$ must be a branch of $J_F^{2/p}$ on $D_1$.
			The last statement is a direct consequence of Lemma \ref{lem: F in another form} (2),
			$$\phi_0(z)T\phi(F(z))=T\phi_0(F(z))\phi(z),\; \forall \phi\in A^p(D_1), z\in D_1.$$
			As $\phi_0(z)=\frac{J_F^{2/p}(z)}{\lambda}$ and $T\phi_0=1$,
			so we get
			$$T\phi(F(z))J_F(z)^{2/p}=\lambda\phi(z),\ \forall \phi\in A^p(D_1), z\in D_1.$$

		\end{proof}

		\section{A counter-example to Conjecture \ref{conj:A^p complete implies iso} if $p>2$}
		The aim of this section is to give an example to show that Conjecture \ref{conj:A^p complete implies iso} could be false if we drop the assumption that $p\in(0,2)$.
		The construction is as follows.
		
		Let $k\geq 1$ be an integer, and consider the map
		$$f_k:\mathbb{C}^2\rightarrow\mathbb{C}^2,\ (z_1,z_2)\mapsto (z_1,z^k_1z_2).$$
		Then $f_k$ gives a biholomorphic map from $\mathbb{C}^*\times \mathbb{C}$ onto $\mathbb{C}^*\times \mathbb{C}$,
		with inverse given by 
		$$g_k(w_1,w_2)=(w_1,w_1^{-k}w_2).$$
		Let $\mathbb{B}$ be the unit ball in $\mathbb{C}^2$, and $\mathbb{B}'=\mathbb{B}\setminus\{z_1=0\}$.
		Set $$D_1=\mathbb{B}\times f_k(\Delta^*\times \Delta),\ D_2=f_k(\mathbb{B}')\times\Delta^2,$$
		where $\Delta$ is the unit disc in $\mc$.
		Note that $$f_k(\Delta^*\times \Delta)=\{(w_1,w_2):|w_2|<|w_1|^k<1\}$$
		and
		$$f_k(\mathbb{B}')=\{(w_1,w_2):|w_2|<|w_1|^k\sqrt{1-|w_1|^2}\}.$$
		It follows that $\mathring{\overline{D_1}}=D_1$ and $\mathring{\overline{D_2}}=D_2$. 
		Therefore, both $D_1$ and $D_2$ are $A^p$-complete for any $p>0$.
		
		Consider the maps
		$$F:D_1\rightarrow \mathbb{C}^4,\ (z_1,z_2,z_3,z_4)\mapsto (w_1,w_2,w_3,w_4):=(z_1,z_1^kz_2,z_3,z_3^{-k}z_4),$$
		and
		$$G:D_2\rightarrow \mathbb{C}^4,\ (w_1,w_2,w_3,w_4)\mapsto (z_1,z_2,z_3,z_4):=(w_1,w_1^{-k}w_2,w_3,w_3^kw_4).$$
		
		Let $D_1'=D_1\setminus\{z_1=0\}$ and $D_2'=D_2\setminus\{w_3=0\}$.
		Then 
		$$F|_{D_1'}:D_1'\rightarrow D_2'$$
		is biholomorphic, whose inverse is $G|_{D_2'}$.
		We have 
		$$J_F^{-1}(0)=\{J_F(z)=(z_1z_3^{-1})^k=0\}=\{z\in D_1: z_1=0\},$$
		$$J_G^{-1}(0)=\{J_G(w)=(w^{-1}_1w_3)^k=0\}=\{w\in D_2: w_3=0\},$$
		and
		$$F(\{z\in D_1:\ z_1=0\})\subset \partial D_2,\  G(\{w\in D_2:\ w_3=0\})\subset \partial D_1.$$
		It follows that that $D_1$ and $D_2$ are domains of BBDT (see Definition \ref{def:BBDP}).
		
		Let $k$ and $m$ two positive integers, such that $p:=\frac{2k}{m}$ is not an even integer.
		We can take $J_F^{2/p}(z)=(z_1z_3^{-1})^m$ (resp. $J_G^{2/p}(w)=(w_1^{-1}w_3)^m$) as a single-valued branch of $J_F^{2/p}$ on $D'_1$ (resp. $J_G^{2/p}$ on $D'_2$).
		Then the map
		$$T:A^p(D_1')\rightarrow A^p(D_2'):T(\phi)(w):=\phi(G(w))J_G^{2/p}(w)$$
		is a linear isometry between $A^p(D_1')$ and $A^p(D_2')$.
		
		If $p\geq 2$, then any $\phi \in A^p(D_j')$ can be extended to a holomorphic function, say $\tilde{\phi}$, on $D_j$ with 
		$$\int_{D_j'}|\phi|^p=\int_{D_j}|\tilde\phi|^p,\ j=1,2.$$ 
		Therefore we can identify  $ A^p(D_j')$ with $A^p(D_j)$ for $j=1,2$, 
		and view $T$ as a linear isometry from $ A^p(D_1)$ to $ A^p(D_2)$.

		\emph{Claim:} $D_1$ and $D_2$ are not biholomorphic.
		
	   It follows that the Conjecture \ref{conj:A^p complete implies iso} could be false if $p>2$.

		We now prove the claim.
		It suffices to show that the dimension of the automorphism groups $Aut(D_1)$ and $Aut(D_2)$ are not equal.
		By definition, we have 
		$$D_1\cong \mathbb B\times \Delta^*\times\Delta,\ D_2\cong\mathbb B'\times\Delta\times\Delta.$$
		By {\cite[Corollary (5.4.11)]{Kob98}}, we just need to show 
		\begin{equation}\label{eqn:dimension of aut group}
			\dim Aut(\mathbb B)+ \dim Aut(\Delta^*)+ \dim Aut(\Delta)\neq \dim Aut(\mathbb B') + \dim Aut(\Delta)+\dim Aut(\Delta).
		\end{equation}
		The left hand side of \eqref{eqn:dimension of aut group} is 12, and the right hand side of \eqref{eqn:dimension of aut group} is $ \dim Aut(\mathbb B')+6$.
		So it suffices to show that $ \dim Aut(\mathbb B')<6$.
		By Riemann's removable singularity theorem, we have 
		$$ \dim Aut(\mathbb B')=\{f\in Aut(\mathbb B)|f(H)=H\},$$
		where $H=\{(z,w)\in \mathbb B|z=0\}$.
		Note that a unitary map preserving $H$ must be diagonal, and a map $f\in Aut(\mathbb B)$ preserving $H$ must map the origin to some point in $H$,
		it follows that $$\dim Aut(\mathbb B')\leq 2+2=4.$$
		We are done.

		\bibliographystyle{amsplain}

	\end{document}